\documentclass[12pt,reqno]{amsart}
\usepackage[headings]{fullpage}
\usepackage{amssymb,amsmath,mathtools,bbm,tikz,tikz-cd,color
}
\usepackage{stmaryrd}
\usepackage{stackengine}
\usepackage[all,cmtip]{xy}
\usepackage{url}
\usepackage{dsfont}
\usepackage{nicematrix}
\usepackage{fancybox}
\usepackage{adjustbox}
\usetikzlibrary{positioning}
\usetikzlibrary{calc}
\usetikzlibrary{decorations.markings}
\usetikzlibrary{arrows.meta}
\usetikzlibrary{calc}
\usetikzlibrary{decorations.pathmorphing}
\usetikzlibrary{decorations.pathreplacing}

\tikzstyle over=[preaction={draw,line width=6pt,white}]
\tikzstyle oversmall=[preaction={draw,line width=3pt,white}]
\tikzset{every path/.style={thick}
}
\tikzset{baseline={([yshift=-.5ex]current bounding box.center)} }

\tikzset{
    doubled/.style={
        preaction={draw=white, line width=4pt},
        preaction={draw=white, line width=4pt, transform canvas={xshift=.5cm}},
        postaction={draw, transform canvas={xshift=.5cm}}
    }
}
\tikzstyle{point}=[postaction={decorate,decoration={markings,
		mark=at position #1 with {\arrow{>}}}}]

\tikzstyle longleadsto=[
{decorate, decoration={
		zigzag,
		segment length=4,
		amplitude=.9,
		post=lineto,
		post length=3pt,
	}
}
]

\allowdisplaybreaks

\usepackage[bookmarks=true,%
colorlinks=true,%
linkcolor=blue,%
citecolor=blue,%
filecolor=blue,%
menucolor=blue,%
urlcolor=blue,%
breaklinks=true]{hyperref}
\usepackage{slashed}    

\usepackage{verbatim}
\usepackage[normalem]{ulem}
\usepackage{diagbox}

\usetikzlibrary{fit}

\newtheorem{theorem}{Theorem}[section]
\theoremstyle{definition}

\newtheorem{lemma}[theorem]{Lemma}

\newtheorem{rem}[theorem]{Remark}
\newtheorem{corollary}[theorem]{Corollary}
\newtheorem{conjecture}[theorem]{Conjecture}

\newenvironment{remark}
  {\pushQED{\qed}\rem}
  {\popQED\endrem}

\def\tr{\mathrm{tr}}

\def\BZ{\mathbbm Z}

\def\BC{\mathbbm C}

\def\la{\langle}
\def\ra{\rangle}

\def\a{\alpha}

\def\be{\begin{equation}}
	\def\ee{\end{equation}}

\def\End{\mathrm{End}}
\def\id{\mathrm{id}}

\def\LG{\mathrm{LG}}

\def\br#1{ \{ #1 \}}

\def\Uq{\mathrm{U}_q}

\makeatletter
\def\namedlabel#1#2{\begingroup
	#2%
	\def\@currentlabel{#2}%
	\phantomsection\label{#1}\endgroup
}
\makeatother

\makeatletter

\renewcommand\thepart{\@Roman\c@part}%
\renewcommand\part{%
	\if@noskipsec \leavevmode \fi
	\par
	\addvspace{6.7ex}%
	\@afterindentfalse
	\secdef\@part\@spart}
\def\@part[#1]#2{%
	\ifnum \c@secnumdepth >\m@ne
	\refstepcounter{part}%
	\addcontentsline{toc}{part}{Part~\thepart.\ #1}%
	\else
	\addcontentsline{toc}{part}{#1}%
	\fi
	{\parindent \z@ \raggedright
		\interlinepenalty \@M
		\normalfont
		\ifnum \c@secnumdepth >\m@ne
		\centering\large\scshape \partname~\thepart.%
		\hspace{1ex}%
		\fi%
		\large\scshape #2%
		\markboth{}{}\par}%
	\nobreak
	\vskip 4.7ex
	\@afterheading}
\def\@spart#1{
	\refstepcounter{part}%
	\addcontentsline{toc}{part}{#1}%
	{\parindent \z@ \raggedright
		\interlinepenalty \@M
		\normalfont
		\centering\large\scshape #1\par}%
	\nobreak
	\vskip 4.7ex
	\@afterheading}
\renewcommand*\l@part[2]{%
	\ifnum \c@tocdepth >-2\relax
	\addpenalty\@secpenalty
	\addvspace{0.75em \@plus\p@}%
	\begingroup
	\parindent \z@ \rightskip \@pnumwidth
	\parfillskip -\@pnumwidth
	{\leavevmode
		\normalsize \bfseries #1\hfil \hb@xt@\@pnumwidth{\hss #2}}\par
	\nobreak
	\if@compatibility
	\global\@nobreaktrue
	\reverypar{\global\@nobreakfalse\reverypar{}}%
	\fi
	\endgroup
	\fi}

\def\l@subsection{\@tocline{2}{0pt}{2pc}{6pc}{}}
\makeatother

\def\AS{\mathrm{AS}}

\renewcommand{\ll}{%
	{%
		\tikz[scale=0.2]{
			\draw[line width=0.5pt] (0,0) to (0,2);
			\draw[line width=0.5pt] (1,0) to (1,2);
		}%
	}%
}

\newcommand{\cc}{%
	{%
		\tikz[scale=0.2]{
			\draw[line width=0.5pt] (0,0) to (0,.25);
			\draw[line width=0.5pt] (1,0) to (1,.25);
			\draw[line width=0.5pt] (0,.25) arc(180:0:0.5);
			\draw[line width=0.5pt] (0,2) to (0,1.75);
			\draw[line width=0.5pt] (1,2) to (1,1.75);
			\draw[line width=0.5pt] (0,1.75) arc(180:360:0.5);
		}%
	}%
}

\newcommand{\xx}{%
	{%
		\tikz[scale=0.2]{
			\draw[line width=0.5pt] (0,0) to[out=90,in=-90] (1,1) to[out=90,in=-90] (0,2);
			\draw[line width=0.5pt] (1,0) to[out=90,in=-90] (0,1) to[out=90,in=-90] (1,2);
		}%
	}%
}

\newcommand{\llsub}{%
	\mathbin{%
		\tikz[baseline=-0.6ex, scale=0.15]{
			\draw[line width=0.5pt] (0,0) to (0,2);
			\draw[line width=0.5pt] (1,0) to (1,2);
		}%
	}%
}

\newcommand{\ccsub}{%
	\mathbin{%
		\tikz[baseline=-0.6ex, scale=0.15]{
			\draw[line width=0.5pt] (0,0) to (0,.25);
			\draw[line width=0.5pt] (1,0) to (1,.25);
			\draw[line width=0.5pt] (0,.25) arc(180:0:0.5);
			\draw[line width=0.5pt] (0,2) to (0,1.75);
			\draw[line width=0.5pt] (1,2) to (1,1.75);
			\draw[line width=0.5pt] (0,1.75) arc(180:360:0.5);
		}%
	}%
}

\newcommand{\xxsub}{%
	\mathbin{%
		\tikz[baseline=-0.6ex, scale=0.15]{
			\draw[line width=0.5pt] (0,0) to[out=90,in=-90] (1,1) to[out=90,in=-90] (0,2);
			\draw[line width=0.5pt] (1,0) to[out=90,in=-90] (0,1) to[out=90,in=-90] (1,2);
		}%
	}%
}

\begin{document}
	
	\title[Detecting Causality with the Links--Gould Polynomial]{Detecting Causality with the Links--Gould Polynomial}
	
	\author[V. Chernov]{Vladimir Chernov}
	\address{Department of Mathematics, Dartmouth College, Hanover, NH
		03755, USA}
	\email{Vladimir.Chernov@dartmouth.edu}

	\author[M. Harper]{Matthew Harper}
	\address{Department of Mathematics, Michigan State University, East Lansing, Michigan, USA}
	\email{mrhmath@proton.me}
	
	\author[B.-M. Kohli]{Ben-Michael Kohli}
	\address{Section de Math\'ematiques, Universit\'e de Gen\`eve \\
		rue du Conseil-G\'en\'eral 7-9, 1205 Gen\`eve, Switzerland}
	\email{bm.kohli@protonmail.ch}

	
	\date{\today }
	
	\begin{abstract}
		The conjectures of Low~\cite{Low0} and Natario--Tod~\cite{NatarioTod}, and Penrose's question on Arnold's Problem list~\cite{ArnoldProblem, ArnoldProblemBook} ask if causality in spacetimes can be formulated in terms of linking of spheres of light rays in the manifold of all light rays. For $(2+1)$-dimensional spacetimes, this link happens in the manifold coverable by a solid torus $S^1\times \mathbb R^2$. This was solved positively by Chernov and Nemirovski~\cite{CNGAFA, CNGT}, see also~\cite{CHigherDimensions}, which raises the question of which link invariants can be used to study causality. 
		Chernov, Martin and Petkova~\cite{CMP} proved that Heegaard--Floer and Khovanov homology completely capture causality. 
		
		Allen--Swenberg~\cite{AS} conjectured that the Jones polynomial, which is obtained as an alternating Euler characteristic from Khovanov homology, is also sufficient. But they constructed complicated examples of links $\AS(n)_{n=1}^{\infty}$ that suggest that the Alexander--Conway polynomial -- which is the Euler characteristic of Heegaard--Floer homology -- is not enough.
		
		The Links--Gould polynomial is a quantum invariant that specializes to the classical Alexander--Conway polynomial in two different ways~\cite{Ishii,Kohli,KPM} and somewhat surprisingly inherits some of its characteristic classical features, see~\cite{Kohli-Tahar, LNVdV25, HKST}. We show that it distinguishes all the Allen-Swenberg links from the link of causally unrelated events and hence detects causality in all known examples where the Alexander--Conway polynomial is not sufficient. We conjecture that it completely captures causality. The work on the categorification of the Links--Gould Polynomial is an ongoing and hard problem, and it is not a subject of this paper.
		
		As a corollary, we also compute the Seifert genus of all Allen--Swenberg links.
		
	\end{abstract}
	
	\maketitle
	
	{\footnotesize
		\tableofcontents
	}
	
	
	\section{Introduction}

	\subsection{Causality and Linking}
	A {\it spacetime\/} $X$ is a time-oriented Lorentz manifold with signature $(+, \cdots, +, -)$. The spacetime is \emph{globally hyperbolic} if it has a Cauchy surface, i.e.~a subset $\Sigma$ such that every maximal causal curve $\gamma$ -- a curve with $\gamma'(t)\cdot \gamma'(t)\leq 0$ -- intersects it exactly once, see Hawking and Ellis~\cite[pp.~211-212]{HE}. Such curves represent particles moving not faster than light.
	The Cauchy surface can be taken to be smooth and spacelike, i.e.~such that the restriction of the Lorentz metric to it is Riemann; and this is where one gives the initial conditions for the Einstein equations. A globally hyperbolic spacetime is diffeomorphic to $\Sigma \times \mathbb R$~\cite{BeSa3, BeSa2, BeSa1} which strengthens the similar and much earlier homeomorphism result of Geroch~\cite{Geroch}.
	
	{\it Globally hyperbolic spacetimes\/} are the most important and studied class of spacetimes. One of the versions of the Strong Cosmic Censorship Conjecture of Penrose~\cite{Penrose} is that all physically relevant spacetimes are like this, that is, if you delete the black holes along their horizons. Nothing can escape black holes, so what happens inside is not relevant to us who live outside of them.
	
	An alternative and equivalent definition of a globally hyperbolic spacetime is that it does not have closed causal curves (i.e. there is no time travel) and that it does not have naked singularities (i.e.~the intersection of the causal future and causal past of any two points is compact and there are no particles that disappear from the view point of the observer), see Bernal--Sanchez~\cite{BeSaCausal},

	Two points (events) $x,y\in X$ are \emph{causally related} if there is a curve $\gamma$ connecting them with $\gamma'(t)\cdot \gamma'(t)\leq 0$, i.e.~if one can get from one point to the other moving not faster than the light speed. Such curves are called {\it causal curves.\/}
	As was shown by Low~\cite{Low0}, the space $\mathcal N_X$ of future-directed unparameterized light rays in $X$ can be naturally identified with the total space $ST^*\Sigma$ of the spherical cotangent bundle of any spacelike Cauchy surface $\Sigma$. The set of light rays through a point $x$ is then identified with the sphere $S_x\subset N_X$ called the \emph{sky of $x$}. Thus, the space $\mathcal N_X$ is a natural contact manifold and $S_x$ is a Legendrian sphere, see~\cite{LowLegendrian} and~\cite{NatarioTod}.
	
	\begin{remark} 
		Two events $x,y$ are {\it chronologically related} if there is a  curve $\gamma(t)$ between them such that the velocity vector along it is everywhere timelike i.e. $\gamma'(t)\cdot \gamma'(t)<0.$ Such curves are called {\it timelike.\/} It is well known that if two events are causally but not chronologically related, then the two events belong to the common light ray between $x$ and $y$. Thus, the sky link $(S_x, S_y)$ is singular, and the two skies have a double point. Such links are called nontrivial by definition.
	\end{remark}
	
	Assuming that the Cauchy surface $\Sigma$ is not homeomorphic  to $\mathbb R P^2$ or $S^2$, Nemirovski and Chernov proved the Low Conjecture~\cite{Low0, Low1, Low3, LowLegendrian}, which states that  events $x,y$ in a $(2+1)$-dimensional globally hyperbolic spacetime $X$ (with $\Sigma \neq S^2, \, \mathbb R P^2$) are causally related if and only if their skies $S_x, S_y$ are linked in $\mathcal N_X$, see \cite{CNGAFA}. This statement is false for $\Sigma=S^2, \, \mathbb R P^2$~\cite{CRCMP}, and in general when a higher dimensional $\Sigma$ admits a structure of a $Y_x^{\ell}$ manifold, see~\cite{Besse}. This also answered Penrose's question from the Arnold problem list~\cite{ArnoldProblemBook}.
	Here \emph{linked} means that  the pair is {\bf not} isotopic to a pair of fibers of the $S^1$-bundle $ST^*\Sigma\to \Sigma$ or that $S_x, S_y$ intersect, meaning that $x,y$ are on a common light ray. We can show that this does not depend on the choice of a Cauchy surface $\Sigma$ and therefore on the identification $\mathcal N_X=ST^*\Sigma.$
	
	\begin{remark}
		Similar conjectures and statements were proved in dimensions $(3+1)$ and above using Legendrian linking~\cite{CNGAFA, CNGT, CHigherDimensions}, thus answering the Legendrian Low Conjecture of Natatrio and Tod~\cite{NatarioTod}. Note also that for all dimensions if one considers the Legendrian link $(S_x, S_y)$ rather than the topological link and the events $x,y$ to be causally related, then one can tell which of the two events is in the future of the other. This is because for the future event there is a so called nonnegative Legendrian isotopy to the one in the past, but not the other way round, see~\cite{CNGAFA, CNGT}. Telling the past event from the future event if one looks at the topological link of skies is not possible.
	\end{remark}

	For the case where the Cauchy surface of $X$ is homeomorphic to $\mathbb R^2$, these skies are unlinked parallel circles in the solid torus $\mathcal N_X=S^1\times \mathbb R^2$, each isotopic to the longitude of the solid torus. The solid torus can be represented as $\mathbb R^3$ minus a trivial knot, so the two component link of two parallel longitudes in the solid torus corresponds to a three component link in $\mathbb R^3$ that is a connected sum of two Hopf links.
	
	Thus, the question of studying causality in $(2+1)$-dimensional globally hyperbolic spacetimes with Cauchy surface $\Sigma\neq S^2, \, \mathbb R P^2$ is now a question about finding link invariants that distinguish $3$-component links from the connected sum of two Hopf links $H\#H$.
	
	Chernov, Martin and Petkova~\cite{CMP} showed that Heegaard--Floer~\cite{HF1, HF2} and Khovanov homology theories~\cite{Khovanov} completely capture causality in such spacetimes. This raised the question of whether their Euler characteristics, the Alexander--Conway and Jones polynomials, capture causality as well. This was thoroughly studied by Allen and Swenberg~\cite{AS} who conjectured that the Jones polynomial completely captures causality. They also found a sequence of 3-component links $\AS(n)_{n=1}^{\infty}$ that look like a link of two skies but are not distinguishable  by the Alexander--Conway polynomial from $H\#H$. See Figure~\ref{fig:FIGAS} for a representation of the first Allen--Swenberg link. Note that there are no known examples of globally hyperbolic spacetimes where $\AS(n)$ links indeed arise as links corresponding to skies, and there are no other known examples where the Alexander--Conway polynomial does not capture causality.

    \begin{figure}[h!]
		\begin{center}
			\adjustbox{trim=0cm 0cm 0cm 0cm}{
				\includegraphics[height=0.45\textheight]{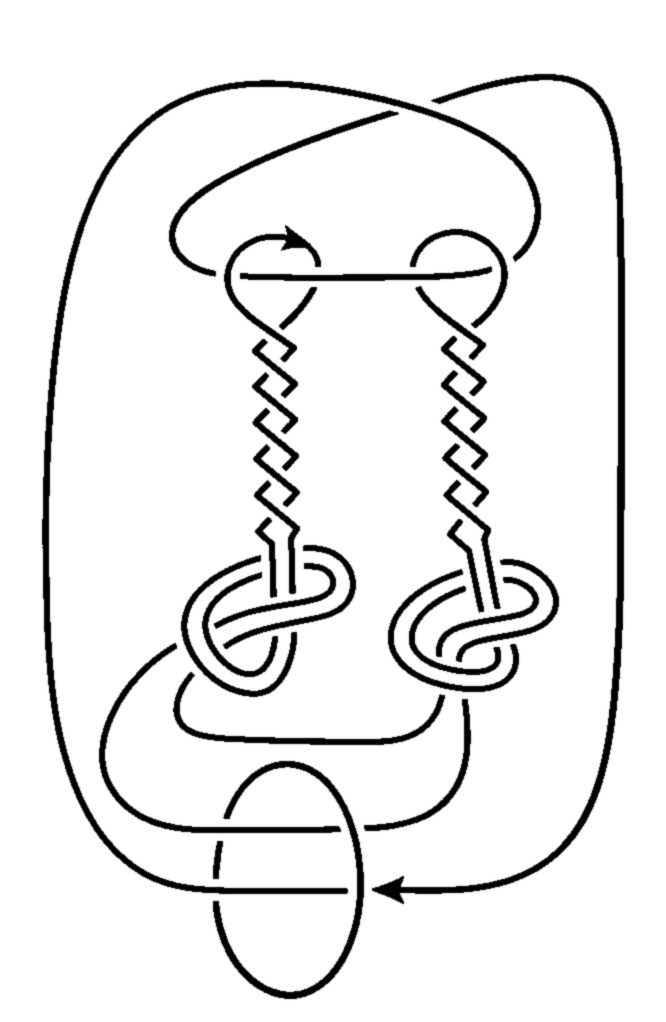}
                \includegraphics[height=0.45\textheight]{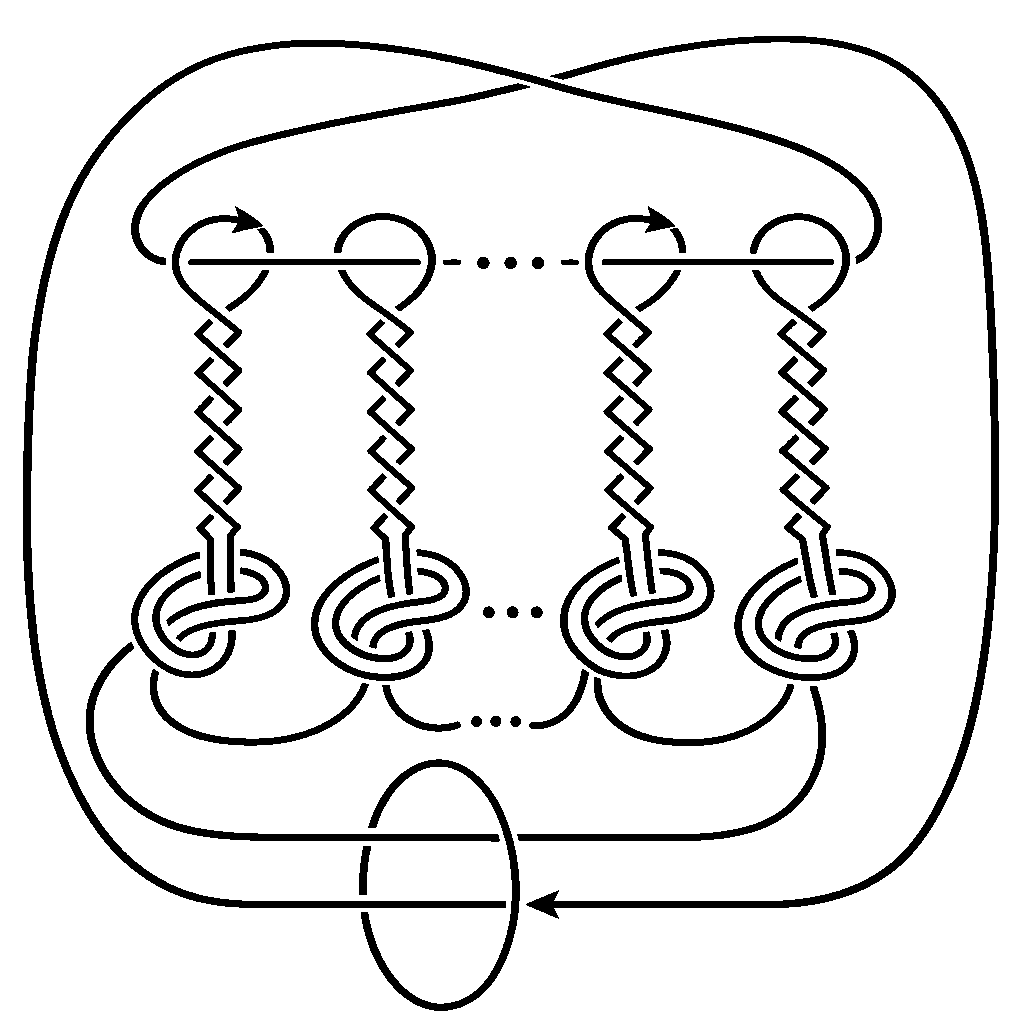}
			}
			\caption{The first and $n$-th Allen--Swenberg links $\AS(1)$ and $\AS(n)$. Note that there are $2n$ twist regions and doubled trefoils in $\AS(n)$.} \label{fig:FIGAS}
		\end{center}
	\end{figure} 
	
	In this work we study the Links--Gould polynomial $\LG$, a two-variable generalization of the Alexander--Conway polynomial. We show that $\LG$ distinguishes all Allen--Swenberg links from the causally unrelated situation $H\#H.$ Thus, $\LG$ completely captures causality in all known examples of $2+1$-dimensional globally hyperbolic spacetimes.

	\begin{theorem}\label{thm.main}
        The Links--Gould polynomial distinguishes all Allen-Swenberg links from the connected sum of Hopf links.
	\end{theorem}

    The theorem is proven by comparing the spans of $\LG_{\AS(n)}$ and $\LG_{H\#H}$ rather than the entire polynomials themselves. Using the multiplicativity under connected sum property of $\LG$, $\LG_{H\#H}=\LG_H^2=\left(\frac{\br{\a}\br{\a+1}}{q}\right)^2$ whose span is $8$. To determine the span for $\AS(n)$, we establish a recursion satisfied by the extremal coefficients of the link polynomials based on the skein theory of the $V(0,\a)\otimes V(0,\a)^*$, where the representation $V(0,\a)$ is the $U_q(\mathfrak{sl}(2|1)$ Kac module used to defined the Links--Gould polynomial. We find 
    \begin{align}
        \mathrm{span}_s(\LG_{\AS(n)}) = 4\cdot(4n+2)>\mathrm{span}_s(\LG_{H\#H}),
    \end{align}
    hence, $\LG$ distinguishes all $\AS(n)$ from $H\#H$.

    The Low conjecture now implies the Links--Gould polynomial detects causality for every sky-like link for which the Alexander--Conway polynomial is known to fail. Given that the Links--Gould polynomial admits specializations to the Alexander polynomial, it will detect causality whenever the Alexander polynomial does.
    \begin{corollary}
         The Links--Gould polynomial completely captures causality in all known examples of globally hyperbolic spacetimes with Cauchy Surface~$\Sigma\neq S^2, \, \mathbb R P^2$.
    \end{corollary}
    
    As a corollary to our span computation, the machinery of \cite{Kohli-Tahar, LNVdV25} allows us to compute the genus of all Allen--Senberg links.

    \begin{corollary}
		For all $n \in \mathbb{Z}_{\geq 1}$, $\mathrm{genus}(\AS(n)) = 2 n.$
	\end{corollary}
	
	\begin{conjecture}
		We conjecture that $\LG$ completely captures causality in such spacetimes. This conjecture is in the spirit of the Allen--Swenberg conjecture for the Jones polynomial~\cite{AS}. Indeed $\LG$ does not have the deficiencies found by Allen and Swenberg for the Alexander--Conway polynomial, which is a specialization of $\LG$.
	\end{conjecture}
	
	\begin{remark}
		Note that the cover of Cauchy surfaces $\mathbb R^2$ induces a cover of the spherical cotangent bundles $ST^*\mathbb R^2\to ST^*\Sigma^2$. Under this coverage, the trivial link consisting of two fibers over two distinct points of $\mathbb R^2$ goes to a similar link consisting of two fibers over two distinct points of $\Sigma$. Thus links corresponding to the skies of two causally unrelated events go to the links corresponding to the skies of two causally unrelated events. This means that in the case where the Cauchy surface $\Sigma\neq S^2, \, \mathbb RP^2$ completely captures causality, it is enough to look at the $3$-component links in $\mathbb R^3$ as we do here, and as was done in~\cite{CMP}.
	\end{remark}
	
	\begin{remark} 
		For spacetimes of dimensions $3+1$ and above there currently are no known invariants that can completely capture causality and the only easily computable invariant is the affine linking number of Chernov--Rudyak~\cite{CRCMP} which is a topological -- rather than Legendrian -- causality detection invariant.
	\end{remark}

	It would be interesting to categorify the Links–Gould polynomial and explore whether it can fully capture causality, in the spirit of~\cite{CMP}. However, this would require substantially new ideas and present many technical challenges.
    
	\begin{remark}
	Some interesting research experiences for high school students have been carried out through the Horizon/Sunrise program under the supervision of Chernov, Zappala, and Maguire. See the works relating quandles and causality by Leventhal~\cite{Leventhal} for affine Alexander quandles, by Chen~\cite{Chen} for medial and affine quandles (in particular correcting computational errors in Leventhal's work for the affine quandle case), Jain~\cite{Jain} and Baxshilloyev~\cite{Bax} for symplectic quandles, by Fan~\cite{Fan} for some non-affine Alexander quandles, and by Wilson~\cite{WilsonQuandle} for disconnected quandles.
    
    Note that quandle invariants are computationally hard objects and not as efficient as the Jones and Links--Gould polynomials. 
Also there is no known relationship between values of Alexander--Conway polynomial and values of quandle coloring invariants. The Alexander--Conway polynomial is a specification of the Links--Gould polynomial and hence there is an obvious hope supported by our results that Links--Gould can solve all the causality related questions that Alexander--Conway can not solve.
	\end{remark}
	
	\subsection{Organization of the paper} In the following Section~\ref{sec.LG} we define $\LG$ and introduce its main properties. Then we give an algorithm to compute $\LG$ on Allen--Swenberg links, which we formally state as Theorem \ref{thm.algorithm}. We use the theorem to determine the leading and trailing terms of $\LG_{\AS(n)}$ for all $n$ in Corollary \ref{cor.leading}. This proves that $\LG$ distinguishes all Allen-Swenberg links from the link of causally unrelated events, and we also determine the $3$-genus of these links.

	\subsection*{Acknowledgements}
	BMK would like to thank Dartmouth College for its hospitality during his stay in Hanover in February 2026, when this project was initiated. MH was partially supported through the NSF-RTG grant
	\#DMS-2135960. The work of BMK is partially supported by the SNSF research program NCCR The Mathematics of Physics (SwissMAP), and the SNSF grant no. 200021-232258. This work was supported in part through computational resources and services provided by the Institute for Cyber-Enabled Research at Michigan State University. The authors are listed in alphabetical order. All authors contributed equally to this work, with each making distinct and complementary contributions. The authors thank an anonymous referee for their suggestions. 

\subsection*{Conflict of interest}
To the best of our knowledge, the authors are not aware of any conflict of interest.

\subsection*{Data availability} Relevant code for this article is available at \cite{Github} and \cite{LG-tangle-computer}.
		
	\section{Computing the Links--Gould polynomial of an Allen--Swenberg link}\label{sec.LG}
	
	\subsection{The Links--Gould polynomial: definition and useful properties}
	
	The Links--Gould polynomial $\LG$ is derived from the 4-dimensional irreducible representation of the unrolled quantum group of $\mathfrak{sl}(2|1)$ with highest weight $(0,\a)$ where $\a\in\BC$. We denote the quantum group by $\Uq$ and this 4-dimensional representation by $V(0,\a)$. We will assume that $\alpha$ is generic in the sense that $\{\a\}\{-1-\alpha\}\neq 0$ and that $q\in\BC$ is not a root of unity. With this assumption on $\a$, the representation $V(0,\alpha)$ is irreducible. See~\cite{GHKW} for our conventions for the quantum group and preferred basis of these representations. We will write $s=q^\a$ and $\br{x}=q^x-q^{-x}$.

    We also recall that $\LG$ is computed from the modified trace construction of \cite{GPT, GeerKujawaPatureau}. Because the quantum dimension of $V(0,\a)$ is zero, the invariant of a link colored by this representation is computed by cutting an arbitrary component and applying the Reshetikhin--Turaev functor (RT) to the resulting 1-tangle. Irreducibility of $V(0,\a)$ will imply that the image of such a 1-tangle is a scalar matrix, and the Links--Gould polynomial is that associated scalar. Since the components of links considered in the present article are all colored by $V(0,\a)$ it is not necessary to consider the modified dimension function for the category of $\Uq$ weight modules in full generality. Instead, we may assume that the modified dimension function is normalized to 1. 
    
	Of particular interest to us in the present paper is the decomposition of the tensor product between $V(0,\a)$ and its dual $V(0,\a)^\ast$ into a direct sum
	\begin{align}\label{eq.decomp}
		V(0,\a)\otimes V(0,\a)^\ast \cong V(1,-1)\oplus P
	\end{align}
	where $P$ is an indecomposable projective module and $V(1,-1)$ is an irreducible representation with highest weight $(1,-1)$. More precisely, $P$  is the projective cover of the trivial representation, the tensor unit, $\mathbbm{1}$. We have shown in~\cite{GHKW} that $\End_{\Uq}(P)\cong \BC[x]/(x^2)$ and therefore $\End_{\Uq}(V(0,\a)\otimes V(0,\a)^\ast)$ is 3-dimensional. 
    
    Using the Reshetikhin--Turaev functor, tangles whose upper and lower boundaries are $\uparrow\downarrow$ give a diagrammatic presentation of vectors in the above endomorphism space. The tangles
	\begin{align}\label{eqn.basis}
		\begin{tikzpicture}[thick,scale=0.6]
			\draw[->] (0,0) to (0,2);
			\draw[<-] (1,0) to (1,2);
		\end{tikzpicture}
		&&
		\begin{tikzpicture}[thick,scale=0.6]
			\draw (0,0) to (0,.25);
			\draw[<-] (1,0) to (1,.25);
			\draw (0,.25) arc(180:0:0.5);
			\draw[<-] (0,2) to (0,1.75);
			\draw (1,2) to (1,1.75);
			\draw (0,1.75) arc(180:360:0.5);
		\end{tikzpicture}
		&&
		\begin{tikzpicture}[thick,scale=0.6]
			\draw (0,0) to[out=90,in=-90] (1,1);
			\draw[over, <-] (1,0) to[out=90,in=-90] (0,1);
			\draw (0,1) to[out=90,in=-90] (1,2);
			\draw[over,->] (1,1) to[out=90,in=-90] (0,2);
		\end{tikzpicture}
	\end{align}
	are such vectors. Throughout this article, we will identify tangles with their image under the RT functor. For simplicity, we may use the notation $\ll$\,, $\cc$\,, and $\xx$ to express these respective vectors. 
	
	We fix a basis of weight vectors $x_0,\dots, x_3$ for $V(0,\alpha)$, where $x_0$ is of highest weight $(0,\alpha)$ and 
	\begin{align}
		F_2x_0=x_1, && F_1x_1=x_2, && F_2x_2=x_3\,.
	\end{align} 
	See \cite{GHKW} for further details. We also set $y_0,\dots,  y_3$ to be the dual basis for $V(0,\alpha)^*$.	To express the tangles in \eqref{eqn.basis} as explicit morphisms on $V(0,\alpha)\otimes V(0,\alpha)^*$, we use a compressed matrix notation. In this notation, the $(i,j)$-matrix entry specifies the image of the vector $x_i\otimes y_j$. As a shorthand, we will suppress the tensor product symbol from the output notation. The matrix for the action of $\ll$~, the identity on $V(0,\a)\otimes V(0,\a)^*$, has $x_iy_j$ in position $(i,j)$. The other two tangles expressed as matrices in this basis are given as follows.
		\begin{align}
			\cc &= \scalebox{.8}{$\dfrac{x_{0} y_{0} + x_{1} y_{1} + x_{2} y_{2} + x_{3} y_{3}}{s^{2}} \left[\begin{matrix} 1 & 0 & 0 & 0\\0 & - 1 & 0 & 0\\0 & 0 & - q^{-2} & 0\\0 & 0 & 0 & q^{-2}\end{matrix}\right] $}\, .
		\end{align}
	We describe the image of $\xx$ in Appendix \ref{app.doublebraid}. Although linear independence of these maps in $\End_{\Uq}(V(0,\alpha)\otimes V(0,\alpha)^*)$ is evident from the off-diagonal entries of $\xx$\,, we will show in Lemma \ref{lem.tr} that an explicit description of $\xx$ is not necessary to prove linear independence. 
	
	\subsection{Components of endomorphisms}
	Recall our assumption that $\a$ is typical, thus $V(0,\a)$ is irreducible. Fix $f\in \End_{\Uq}(V(0,\a)\otimes V(0,\a)^\ast)$. The goal of this section is to show that the vectors $\ll$\,, $\cc$\,, and $\xx$ form a basis $\mathcal{B}$ of the endomorphism space and to describe a method of determining the components $f_{\llsub}$\,, $f_{\ccsub}$\,, $f_{\xxsub}$ of $f$ in the basis. We then express the clasped $(2,6)$-cablings of the trefoil in this basis.
	
	First, we define three partial trace operations on $f$. The \emph{right trace} $\tr_R$, \emph{top trace} $\tr_T$, and \emph{(negatively) twisted right trace} $\widetilde{\tr}_R$ of $f$ are defined as follows.
	\begin{align}
		\tr_R(f):=
		\begin{tikzpicture}[thick,scale=0.5]
			\draw (0,0) rectangle (2,1) node[pos=.5] {$f$};
			\def\offset{0.2}
			\draw (2-\offset,1) arc(180:0:.5);
			\draw (2-\offset,0) arc(180:360:.5);
			\draw[->] (0+\offset,1) to (0+\offset, 2);
			\draw[point=.5] (0+\offset,-1) to (0+\offset, 0);
			\draw[->] (3-\offset, 0) to (3-\offset, 1);
		\end{tikzpicture}
        \,,
		&&
		\tr_T(f):=
		\begin{tikzpicture}[thick,scale=0.5]
			\draw (0,0) rectangle (2,1) node[pos=.5] {$f$};
			\def\offset{0.2}
			\draw[point=0.5] (0+\offset,1) arc(180:0:1-\offset);
			\draw (2-\offset,0) arc(180:360:.5);
			\draw[point=.5] (0+\offset,-1) to (0+\offset, 0);
			\draw[->] (3-\offset, 0) to (3-\offset, 2);
		\end{tikzpicture}\,,
		&&
		\widetilde{\tr}_R(f):=\tr_R(
		\begin{tikzpicture}[thick,scale=.45]
			\draw[<-] (1,0) to[out=90,in=-90] (0,1);
			\draw[oversmall] (0,0) to[out=90,in=-90] (1,1);
			\draw[->] (1,1) to[out=90,in=-90] (0,2);
			\draw[oversmall] (0,1) to[out=90,in=-90] (1,2);
		\end{tikzpicture}\circ f)=
		\begin{tikzpicture}[thick,scale=0.5]
			\draw (0,0) rectangle (2,1) node[pos=.5] {$f$};
			\def\offset{0.2}
			\draw (2-\offset,1) to[out=90,in=-90] (\offset,2);
			\draw[oversmall] (\offset,1) to[out=90,in=-90] (2-\offset,2);
			\draw[->] (2-\offset,2) to[out=90,in=-90] (\offset,3) to (\offset,3.5);
			\draw[oversmall] (\offset,2) to[out=90,in=180] (2-\offset,3);
			\draw[point=.5] (2-\offset+1,0) to[out=90,in=0] (2-\offset,3) ;
			\draw (2-\offset,0) arc(180:360:.5);
			\draw[point=.5] (0+\offset,-1) to (0+\offset, 0);
		\end{tikzpicture}
	\end{align}
	For each partial trace, notice that the result is a morphism in $\End_{\Uq}(V(0,\a))$. Since we have assumed that $V(0,\a)$ is irreducible, Schur's lemma implies that $\End_{\Uq}(V(0,\a))$ is naturally identified with the ground field. For $g\in\End_{\Uq}(V(0,\a))$, we write $g=\langle g\rangle\cdot \id_{V(0,\a)}$. 
    \begin{lemma}\label{lem.tr}
        The vectors $\ll$\,, $\cc$\,, and $\xx$ form a basis $\mathcal{B}$ of $\End_{\Uq}(V(0,\a)\otimes V(0,\a)^*)$. For any $f\in \End_{\Uq}(V(0,\a)\otimes V(0,\a)^*)$, the components of $f$ in the basis $\mathcal{B}$ are determined by the linear system 
        \begin{align}\label{eq.linsys}
		\begin{bmatrix}
			0&1&\dfrac{\br{\a}\br{\a+1}}{q}
			\\
			1&0&1
			\\
			q\br{\a}\br{\a+1}&1&0
		\end{bmatrix}\cdot 
		\begin{bmatrix}
			f_{\llsub}\\f_{\ccsub}\\f_{\xxsub}
		\end{bmatrix}=\begin{bmatrix}
			\la \tr_R(f)\ra\\\la \tr_T(f)\ra\\
			\left\la \widetilde{\tr}_R(f)\right\ra
		\end{bmatrix}\,.
	\end{align}
    \end{lemma}
    \begin{proof}
    The linear-independence of the vectors follows from the determinant of the proposed linear system. In particular, the determinant $(q+q^{-1})\br{\a}\br{\a+1}$ is nonzero for typical $\a$ and $q^2\neq -1$. These form a basis as the morphism space is 3-dimensional following \eqref{eq.decomp}.
    
    Each trace of each basis tangle is given below. Recall that $s=q^\a$ and that we use the notation $\br{x}=q^x-q^{-x}$. 
	\begin{align}
		\left\langle \tr_R\left(
		\begin{tikzpicture}[thick,scale=0.45]
			\draw[->] (0,0) to (0,2);
			\draw[<-] (1,0) to (1,2);
		\end{tikzpicture}\right)\right\rangle &= 0,
		&
		\left\langle \tr_R\left(\begin{tikzpicture}[thick,scale=0.45]
			\draw (0,0) to (0,.25);
			\draw[<-] (1,0) to (1,.25);
			\draw (0,.25) arc(180:0:0.5);
			\draw[<-] (0,2) to (0,1.75);
			\draw (1,2) to (1,1.75);
			\draw (0,1.75) arc(180:360:0.5);
		\end{tikzpicture}\right)\right\rangle &= 1,
		&
		\left\langle \tr_R\left(\begin{tikzpicture}[thick,scale=.45]
			\draw (0,0) to[out=90,in=-90] (1,1);
			\draw[oversmall, <-] (1,0) to[out=90,in=-90] (0,1);
			\draw (0,1) to[out=90,in=-90] (1,2);
			\draw[oversmall,->] (1,1) to[out=90,in=-90] (0,2);
		\end{tikzpicture}\right)\right\rangle &
		= \frac{\br{\a}\br{\a+1}}{q}
		\\
		\left\langle \tr_T\left(\begin{tikzpicture}[thick,scale=0.45]
			\draw[->] (0,0) to (0,2);
			\draw[<-] (1,0) to (1,2);
		\end{tikzpicture}\right)\right\rangle &= 1,
		&
		\left\langle \tr_T\left(\begin{tikzpicture}[thick,scale=0.45]
			\draw (0,0) to (0,.25);
			\draw[<-] (1,0) to (1,.25);
			\draw (0,.25) arc(180:0:0.5);
			\draw[<-] (0,2) to (0,1.75);
			\draw (1,2) to (1,1.75);
			\draw (0,1.75) arc(180:360:0.5);
		\end{tikzpicture}\right)\right\rangle &= 0,
		&
		\left\langle \tr_T\left(\begin{tikzpicture}[thick,scale=.45]
			\draw (0,0) to[out=90,in=-90] (1,1);
			\draw[oversmall, <-] (1,0) to[out=90,in=-90] (0,1);
			\draw (0,1) to[out=90,in=-90] (1,2);
			\draw[oversmall,->] (1,1) to[out=90,in=-90] (0,2);
		\end{tikzpicture}\right)\right\rangle &= 1
		\\
		\left\langle \widetilde{\tr}_R\left(
		\begin{tikzpicture}[thick,scale=0.45]
			\draw[->] (0,0) to (0,2);
			\draw[<-] (1,0) to (1,2);
		\end{tikzpicture}\right)\right\rangle &= q\br{\a}\br{\a+1},
		&
		\left\langle \widetilde{\tr}_R\left(
		\begin{tikzpicture}[thick,scale=0.45]
			\draw (0,0) to (0,.25);
			\draw[<-] (1,0) to (1,.25);
			\draw (0,.25) arc(180:0:0.5);
			\draw[<-] (0,2) to (0,1.75);
			\draw (1,2) to (1,1.75);
			\draw (0,1.75) arc(180:360:0.5);
		\end{tikzpicture}\right)\right\rangle &= 1,
		&
		\left\langle \widetilde{\tr}_R\left(
		\begin{tikzpicture}[thick,scale=.45]
			\draw (0,0) to[out=90,in=-90] (1,1);
			\draw[oversmall, <-] (1,0) to[out=90,in=-90] (0,1);
			\draw (0,1) to[out=90,in=-90] (1,2);
			\draw[oversmall,->] (1,1) to[out=90,in=-90] (0,2);
		\end{tikzpicture}\right)\right\rangle &= 0
	\end{align}  Applying these same trace operations to $f$, their linearity properties imply that its components are determined by the prescribed linear system. 
    \end{proof}
	   
	The image of the tangles in Figure \ref{fig.tt} under the RT functor was computed using the Michigan State University HPCC using the code at cite \cite{LG-tangle-computer}. Implementing Lemma \ref{lem.tr} via the script at \cite{Github}, we find that their components in the basis $\mathcal{B}$ are:
    
	\begin{align}
		TT^+&=
		\scalebox{.6}{$
			\begin{bmatrix}
				\left(s^{2}+q^{-2}s^{-2}\right) \left(1 - \dfrac{2}{q^{2}} + \dfrac{2}{q^{4}} - \dfrac{2}{q^{6}} + \dfrac{4}{q^{8}} - \dfrac{2}{q^{10}}\right) - 2 + \dfrac{3}{q^{2}} - \dfrac{4}{q^{4}} + \dfrac{4}{q^{6}} - \dfrac{8}{q^{8}} + \dfrac{4}{q^{10}}
				\\
				\left(s^{4}+q^{-4}s^{-4}\right) \left(2 - \dfrac{2}{q^{2}} - \dfrac{4}{q^{6}} + \dfrac{4}{q^{8}}\right) + \left(s^{2}+q^{-2}s^{-2}\right) \left(-5 + \dfrac{4}{q^{2}} + \dfrac{12}{q^{6}} - \dfrac{8}{q^{8}} - \dfrac{4}{q^{10}}\right) + 4 + \dfrac{3}{q^{2}} - \dfrac{6}{q^{4}} - \dfrac{6}{q^{6}} - \dfrac{8}{q^{8}} + \dfrac{16}{q^{10}}
				\\
				\left(s^{2}+q^{-2}s^{-2}\right) \left(\dfrac{2}{q^{2}} - \dfrac{2}{q^{4}} + \dfrac{2}{q^{6}} - \dfrac{4}{q^{8}} + \dfrac{2}{q^{10}}\right) + 1 - \dfrac{4}{q^{2}} + \dfrac{4}{q^{4}} - \dfrac{4}{q^{6}} + \dfrac{8}{q^{8}} - \dfrac{4}{q^{10}}
			\end{bmatrix}$}\,,
		\label{eq.tt+}
		\\
		TT^-&=\scalebox{.6}{
			$\begin{bmatrix}
				(s^{2}+q^{-2}s^{-2}) \left(q^{2} - 2 + \dfrac{2}{q^{2}} - \dfrac{2}{q^{4}} + \dfrac{4}{q^{6}} - \dfrac{2}{q^{8}}\right) - 1 + \dfrac{4}{q^{2}} - \dfrac{4}{q^{4}} + \dfrac{4}{q^{6}} - \dfrac{8}{q^{8}} + \dfrac{4}{q^{10}}
				\\
				(s^{4}+q^{-4}s^{-4}) \left(2 q^{2} - 2 - \dfrac{4}{q^{4}} + \dfrac{4}{q^{6}}\right) + (s^{2}+q^{-2}s^{-2}) \left(- 3 q^{2} - 4 + \dfrac{8}{q^{2}} + \dfrac{4}{q^{4}} + \dfrac{8}{q^{6}} - \dfrac{12}{q^{8}}\right) + q^2 + 10  - \dfrac{6}{q^{2}} - \dfrac{6}{q^{4}} - \dfrac{16}{q^{6}} + \dfrac{8}{q^{8}} + \dfrac{8}{q^{10}}
				\\
				(s^{2}+q^{-2}s^{-2}) \left(2 - \dfrac{2}{q^{2}} + \dfrac{2}{q^{4}} - \dfrac{4}{q^{6}} + \dfrac{2}{q^{8}}\right) - q^{2} - \dfrac{4}{q^{2}} + \dfrac{4}{q^{4}} - \dfrac{4}{q^{6}} + \dfrac{8}{q^{8}} - \dfrac{4}{q^{10}}
			\end{bmatrix}$}\,.
		\label{eq.tt-}
	\end{align}

    \begin{figure}[h!]
    \[TT^+:=
    \begin{tikzpicture}
        \tikzstyle over=[preaction={draw,line width=4pt,white}]
            \draw[->] (0,-.5) to (0,0);
            \draw[<-] (.5,-.5) to (.5,0);
			\draw[doubled] (1,0) to[out=180,in=-90] (0,1); 
            \draw[doubled] (0,0) to[out=90,in=-90] (1,1);
			\draw[doubled] (0,1) to[out=90,in=-90] (1,2); 
            \draw[doubled] (1,1) to[out=90,in=-90] (0,2); 
			\draw[doubled](0,1) to[out=90,in=-90] (1,2); 
				\draw[doubled] (1,2) to[out=90,in=-90] (0,3); 
			\draw[doubled] (0,2) to[out=90,in=180] (1,3);  
			\draw[doubled] (0,2) to[out=90,in=180] (1,3); 
			\draw[doubled] (1,3) to[out=0,in=0] (1,0);
            \draw[doubled] (0,3) to (0,3.05);
            \draw (0.5,3.05) to[out=90,in=-90] (0,3.5);
            \draw[over] (0,3.05) to[out=90,in=-90] (0.5,3.5);
            \draw (0.5,3.5) to[out=90,in=-90] (0,4);
            \draw[over] (0,3.5) to[out=90,in=-90] (0.5,4);
            \draw (0.5,4) to[out=90,in=-90] (0,4.5);
            \draw[over] (0,4) to[out=90,in=-90] (0.5,4.5);
            \draw (0.5,4.5) to[out=90,in=-90] (0,5);
            \draw[over] (0,4.5) to[out=90,in=-90] (0.5,5);
            \draw (0.5,5) to[out=90,in=-90] (0,5.5);
            \draw[over] (0,5) to[out=90,in=-90] (0.5,5.5);
            \draw (0.5,5.5) to[out=90,in=-90] (0,6);
            \draw[over,<-] (0,5.5) to[out=90,in=-90] (0.5,6);
            \draw[point=1] (0.25,6) to[out=180, in=-90] (0,6.5);
            \draw[over] (0,6) to[out=90,in=180] (0.25, 6.25);
            \draw (0.25, 6.25) to[out=0,in=90] (0.5, 6);
            \draw[over,point=.25] (0.5,6.5) to[out=270, in=0] (0.25,6);
        \end{tikzpicture}
        \qquad\qquad
        TT^-:=
		\begin{tikzpicture}
        \tikzstyle over=[preaction={draw,line width=4pt,white}]
            \draw[->] (0,-.5) to (0,0);
            \draw[<-] (.5,-.5) to (.5,0);
			\draw[doubled] (1,0) to[out=180,in=-90] (0,1); 
            \draw[doubled] (0,0) to[out=90,in=-90] (1,1);
			\draw[doubled] (0,1) to[out=90,in=-90] (1,2); 
            \draw[doubled] (1,1) to[out=90,in=-90] (0,2); 
			\draw[doubled](0,1) to[out=90,in=-90] (1,2); 
				\draw[doubled] (1,2) to[out=90,in=-90] (0,3); 
			\draw[doubled] (0,2) to[out=90,in=180] (1,3);  
			\draw[doubled] (0,2) to[out=90,in=180] (1,3); 
			\draw[doubled] (1,3) to[out=0,in=0] (1,0);
            \draw[doubled] (0,3) to (0,3.05);
            \draw (0.5,3.05) to[out=90,in=-90] (0,3.5);
            \draw[over] (0,3.05) to[out=90,in=-90] (0.5,3.5);
            \draw (0.5,3.5) to[out=90,in=-90] (0,4);
            \draw[over] (0,3.5) to[out=90,in=-90] (0.5,4);
            \draw (0.5,4) to[out=90,in=-90] (0,4.5);
            \draw[over] (0,4) to[out=90,in=-90] (0.5,4.5);
            \draw (0.5,4.5) to[out=90,in=-90] (0,5);
            \draw[over] (0,4.5) to[out=90,in=-90] (0.5,5);
            \draw (0.5,5) to[out=90,in=-90] (0,5.5);
            \draw[over] (0,5) to[out=90,in=-90] (0.5,5.5);
            \draw (0.5,5.5) to[out=90,in=-90] (0,6);
            \draw[over,<-] (0,5.5) to[out=90,in=-90] (0.5,6);
            \draw (0,6) to[out=90,in=180] (0.25, 6.25);
            \draw[over,point=1] (0.25,6) to[out=180, in=-90] (0,6.5);
            \draw[point=.25] (0.5,6.5) to[out=270, in=0] (0.25,6);
            \draw[over] (0.25, 6.25) to[out=0,in=90] (0.5, 6);
		\end{tikzpicture}
        \]
        \caption{Positively and negatively clasped antiparallel (2,6)-cablings of the open trefoil.}
        \label{fig.tt}
	\end{figure} 
	
	\subsection{Horizontal concatenation and evaluation}
	We construct the inner $(2,2)$-tangle containing all (2,6)-cablings of trefoils of $\AS(n)$ seen in Figure \ref{fig:FIGAS} via horizontal concatenation, first by concatenating $TT^+$ and $TT^-$ as in the diagram of \eqref{eq.boxtimes} below, followed by concatenating this joined tangle with itself $n-1$ more times. To compute the image of the Reshetikhin--Turaev functor of the result, towards computing the Links--Gould polynomial of $\AS(n)$, it will be useful to carefully describe the components of a pair of tangles under this concatenation operation. 
	
	For $f,g\in \End_{\Uq}(V(0,\a)\otimes V(0,\a)^\ast)$, define
	\begin{align}\label{eq.boxtimes}
		f\boxtimes g:= 
		\begin{tikzpicture}[thick,scale=0.5]
			\draw (0,0) rectangle (2,1) node[pos=.5] {$f$};
			\draw (3,0) rectangle (5,1) node[pos=.5] {$g$};
			\def\offset{0.2}
			\draw[point=.5] (3+\offset,1) .. controls (2.75,1.5) and (2.25,1.5) .. (2-\offset,1);
			\draw[point=.5] (2-\offset,0) .. controls (2.25,-.5) and (2.75,-.5) .. (3+\offset,0);
			\draw[->] (0+\offset, 1) to (0+\offset, 2);
			\draw[->] (5-\offset, 0) to (5-\offset, -1);
			\draw[point=.5] (5-\offset, 2) to (5-\offset, 1);
			\draw[point=.5] (\offset, -1) to (\offset, 0);
		\end{tikzpicture}\in \End_{\Uq}(V(0,\a)\otimes V(0,\a)^\ast)\,.
	\end{align}
	Since $f$ and $g$ themselves are a linear combination of three tangles, we will first consider $f\boxtimes g$ for $f$ and $g$ belonging to our preferred basis, then extend linearly. More precisely, 
	\begin{align}
		f\boxtimes g = \sum_{i,j\in \mathcal{B}
		}f_ig_j\cdot (i\boxtimes j)\,.
	\end{align} 
    Table \ref{table.boxtimes} expresses horizontal concatenation in the basic cases.
	
    \begin{table}[h!]
		\centering 
        \footnotesize
		\begin{NiceTabular}{c|c|c|c}[cell-space-limits=3pt]
			$\boxtimes$ & $\ll$ & $\cc$ & $\xx$ \\ \hline
			
			$\ll$ &
			{\footnotesize $\begin{bmatrix} 0\\0\\0 \end{bmatrix}$} &
			{\footnotesize $\begin{bmatrix} 1\\0\\0 \end{bmatrix}$} &
			{\footnotesize $\begin{bmatrix} \frac{\br{\a}\br{\a+1}}{q} \\ 0\\0 \end{bmatrix}$} \\ \hline
			
			$\cc$ &
			{\footnotesize $\begin{bmatrix} 1\\0\\0 \end{bmatrix}$} &
			{\footnotesize $\begin{bmatrix} 0\\1\\0 \end{bmatrix}$} &
			{\footnotesize $\begin{bmatrix} 0\\0\\1 \end{bmatrix}$} \\ \hline
			
			$\xx$ &
			{\footnotesize $\begin{bmatrix} \frac{\br{\a}\br{\a+1}}{q} \\ 0\\0 \end{bmatrix}$} &
			{\footnotesize $\begin{bmatrix} 0\\0\\1 \end{bmatrix}$} &
			{\footnotesize $\begin{bmatrix}
					(s^{2} + q^{-2}s^{-2})\left(1 + \frac{1}{q^{2}}\right) - \frac{3}{q^{2}} - \frac{1}{q^{4}} \\
					- (s^{4} + q^{-4}s^{-4}) + (s^{2} + q^{-2}s^{-2}) \left(3 + \frac{1}{q^{2}}\right) - 2 - \frac{4}{q^{2}} \\
					2(s^{2} + q^{-2}s^{-2}) - 3 - \frac{1}{q^{2}}
				\end{bmatrix}$}
		\end{NiceTabular}%
		\caption{Components of $\boxtimes$ between pairs of basis elements.}
		\label{table.boxtimes}
	\end{table}
	
		Inserting the $(2,2)$-tangle $(TT^+\boxtimes TT^-)^{\boxtimes n}$ into the diagram of \eqref{eq.exterior} will produce a $(1,1)$-tangle whose closure is $\AS(n)$. Once again, it suffices to consider basis vector $(2,2)$-tangles and extend linearly. For $f\in \End_{\Uq}(V(0,\a)\otimes V(0,\a)^*)$, define
	\begin{align}\label{eq.exterior}
		\AS^*(~f~)=
		\left\la~
		\begin{tikzpicture}[scale=.75]
			\def\offset{0.2}
			\draw (2-\offset,1) to[out = 90, in = 0] (\offset,2);
			\draw[over] (\offset,1) to[out = 90, in = -90] (1,2.5);
            \draw[->] (1,2.5) to (1,2.6);
			\draw (-1-\offset,-1) to[out=90, in=90] (.5+2*\offset,-1);
			\draw[->,over] (-.5,-1) to[out=90,in=180] (\offset,2);
			\draw (-.5,-1) to (-.5, -2);
			\draw[over,point=.5] (\offset, -1) to (\offset,0);
			\draw (\offset,-1) arc(180:360:1-\offset);
			\draw[->] (2-\offset, 0) to (2-\offset,-1);
			\draw[over,point=.6] (-1-\offset,-1) to[out=-90, in=-90] (.5+2*\offset,-1);
			\draw (0,0) rectangle (2,1) node[pos=.5] {$f$};
		\end{tikzpicture}~
		\right\ra\,.
	\end{align} 
    In this way $\LG_{\AS(n)}(s,q)=\AS^*((TT^+\boxtimes TT^-)^{\boxtimes n})$. The polynomials associated to basis vectors are 
	\begin{align}
		\AS^*(~\ll~)=&~(s^{6}+q^{-6}s^{-6}) + (s^{4}+q^{-4}s^{-4})\left(-1 - \frac{1}{q^{2}}\right) 
		\label{eq.ASll}
		\\
		&+ (s^{2}+q^{-2}s^{-2})\left(1 + \frac{2}{q^{2}}\right) - 1 - \frac{2}{q^{2}} - \frac{1}{q^{4}} \, ,\notag
		\\
		\AS^*(~\cc~)=&~(s^{4}+q^{-4}s^{-4}) - (s^{2}+q^{-2}s^{-2})\left(2 + \frac{2}{q^{2}}\right) + 1 + \frac{4}{q^{2}} + \frac{1}{q^{4}} \, ,
		\label{eq.AScc}
		\\
		\AS^*(~\xx~)=&~
		(s^{6}+q^{-6}s^{-6}) \left(3 + \frac{1}{q^{2}}\right) - (s^{4}+q^{-4}s^{-4}) \left(7 + \frac{8}{q^{2}} + \frac{1}{q^{4}}\right) 
		\label{eq.ASxx}
		\\
		&+ (s^{2}+q^{-2}s^{-2})\left(7 + \frac{18}{q^{2}} + \frac{7}{q^{4}}\right) - 3 - \frac{18}{q^{2}} - \frac{17}{q^{4}} - \frac{2}{q^{6}}\, ,\notag
	\end{align}
	which we write as a row vector $\overrightarrow{\AS}^*$. 
	\begin{theorem}\label{thm.algorithm}
		The $n^{\mathrm{\textit{th}}}$ AS link has Links--Gould polynomial
		\[
		\LG_{\AS(n)}(s,q)=\overrightarrow{\AS}^*\cdot (TT^+\boxtimes TT^-)^{\boxtimes n}.
		\]
	\end{theorem}

	\begin{remark}\label{rem.AS1}
    The first AS polynomial $\LG_{\AS(1)}(s,q)$ is computed explicitly as 
	\begin{tiny}
	    
		\begin{align*}
			&(s^{12}+q^{-12}s^{-12}) \left(8 q^{2} + 8 - \frac{8}{q^{2}} - \frac{24}{q^{4}} - \frac{8}{q^{6}} + \frac{32}{q^{10}} + \frac{12}{q^{12}} - \frac{8}{q^{14}} - \frac{8}{q^{16}} - \frac{16}{q^{18}} + \frac{12}{q^{20}}\right) + 
			\notag\\
			&(s^{10}+q^{-10}s^{-10}) \left(- 40 q^{2} - 88 + \frac{12}{q^{2}} + \frac{184}{q^{4}} + \frac{160}{q^{6}} + \frac{24}{q^{8}} - \frac{232}{q^{10}} - \frac{200}{q^{12}} + \frac{20}{q^{14}} + \frac{104}{q^{16}} + \frac{124}{q^{18}} - \frac{24}{q^{20}} - \frac{44}{q^{22}}\right) + 
			\notag\\
			&(s^{8}+q^{-8}s^{-8}) \left(90 q^{2} + 336 + \frac{176}{q^{2}} - \frac{564}{q^{4}} - \frac{900}{q^{6}} - \frac{374}{q^{8}} + \frac{694}{q^{10}} + \frac{1186}{q^{12}} + \frac{260}{q^{14}} - \frac{456}{q^{16}} - \frac{596}{q^{18}} - \frac{176}{q^{20}} + \frac{276}{q^{22}} + \frac{48}{q^{24}}\right) + 
			\notag\\
			&(s^{6}+q^{-6}s^{-6}) \left(- 120 q^{2} - 668 - \frac{866}{q^{2}} + \frac{738}{q^{4}} + \frac{2516}{q^{6}} + \frac{1914}{q^{8}} - \frac{792}{q^{10}} - \frac{3544}{q^{12}} - \frac{2002}{q^{14}} + \frac{800}{q^{16}} + \frac{1884}{q^{18}} + \frac{1184}{q^{20}} - \frac{604}{q^{22}} - \frac{424}{q^{24}} - \frac{16}{q^{26}}\right) + 
			\notag\\
			&(s^{4}+q^{-4}s^{-4}) \left(100 q^{2} + 805 + \frac{1816}{q^{2}} + \frac{216}{q^{4}} - \frac{4024}{q^{6}} - \frac{4980}{q^{8}} - \frac{1100}{q^{10}} + \frac{5824}{q^{12}} + \frac{6436}{q^{14}} + \frac{396}{q^{16}} - \frac{3544}{q^{18}} - \frac{3696}{q^{20}} + \frac{88}{q^{22}} + \frac{1436}{q^{24}} + \frac{228}{q^{26}}\right) + 
			\notag\\
			&(s^{2}+q^{-2}s^{-2}) \left(- 48 q^{2} - 606 - \frac{2154}{q^{2}} - \frac{2082}{q^{4}} + \frac{3290}{q^{6}} + \frac{7788}{q^{8}} + \frac{5154}{q^{10}} - \frac{4404}{q^{12}} - \frac{11360}{q^{14}} - \frac{4914}{q^{16}} + \frac{3516}{q^{18}} + \frac{6540}{q^{20}} + \frac{2676}{q^{22}} - \frac{2284}{q^{24}} - \frac{1076}{q^{26}} - \frac{40}{q^{28}}\right) 
			\notag\\
			&+ 10 q^{2} + 261
			+ \frac{1530}{q^{2}} + \frac{3001}{q^{4}} - \frac{190}{q^{6}} - \frac{7308}{q^{8}} - \frac{8354}{q^{10}} - \frac{1020}{q^{12}} + \frac{10968}{q^{14}} + \frac{10644}{q^{16}} - \frac{56}{q^{18}} - \frac{6704}{q^{20}} - \frac{6280}{q^{22}} + \frac{784}{q^{24}} + \frac{2376}{q^{26}} + \frac{344}{q^{28}} \, .
		\end{align*}
	\end{tiny}
	In particular, the $q=1$ specialization of the Links--Gould polynomial for the first Allen--Swenberg link is
		\begin{align*}
			\LG_{\AS(1)}(s,1)=\frac{(s+1)^4 \cdot (s-1)^4}{s^4} = (s-s^{-1})^4 = \Delta_{H\#H}(s^2)^2.
		\end{align*} 
		This is expected and a confirmation of the result of Allen--Swenberg \cite{AS}.
	\end{remark}

	Given the complexity of the expressions involved, a general closed formula is difficult to write. Nevertheless, the leading and trailing terms of the polynomials are readily determined for all $n$.
	
	In the following corollary, we consider $\LG$ as a Laurent polynomial in $s$ with coefficients in $\BZ[q,q^{-1}]$. By \emph{leading term}, we refer to the leading term of the polynomial in $s$ and take its leading term in the $\BZ[q,q^{-1}]$ coefficient. By \emph{trailing term}, we refer to the trailing term of the polynomial in $s$ and take its leading term in the $\BZ[q,q^{-1}]$ coefficient. We will use $\sim$ to indicate taking the leading term of a polynomial. The result implies Theorem \ref{thm.main}.
	\begin{corollary}\label{cor.leading}
		The leading and trailing terms of $\LG_{\AS(n)}$ are equal to {$s^{4+8n}\cdot q^{2n}\cdot(n+1)\cdot4^n$} and {$s^{-4-8n}\cdot q^{-4-6n}\cdot(n+1)\cdot4^n$}, respectively. Therefore, the span of the Links--Gould polynomial in the variable $s$ is 
		\[\mathrm{span}_s(\LG_{\AS(n)}) = 4\cdot(4n+2)\,.\]
	\end{corollary}
	
	\begin{proof}
		We will prove the first part of the corollary for the leading term. The expression for the trailing term is a consequence of the symmetry $s\mapsto q^{-1}s^{-1}$ in each of the terms involved.
		
		Let us restrict to the leading terms of the expressions $TT^+$, $TT^-$, $\overrightarrow{\AS}^*$ given in Equations \eqref{eq.tt+}, \eqref{eq.tt-}, \eqref{eq.ASll}, \eqref{eq.AScc}, \eqref{eq.ASxx}, as well as those in Table \ref{table.boxtimes}.
		We compute
		\begin{align}
			TT^+\boxtimes TT^- \sim 
			\begin{bmatrix}
				4s^6q^2 \\ 4s^8q^2 \\ 8s^6    
			\end{bmatrix}=:T \, .
		\end{align}
		Moreover for any vector $x\cdot \ll + y\cdot \cc + z\cdot \xx$, we have
		\begin{align}
			\begin{bmatrix}
				x\\y\\z
			\end{bmatrix}
			\boxtimes
			T
			\sim
			\begin{bmatrix}
				4s^8q^2x + 4s^6q^2y + 4s^8q^2z
				\\
				4s^8q^2y - 8s^{10}  z
				\\
				8s^6 y  + 4s^8q^2z 
			\end{bmatrix}\,.
		\end{align}
		After computing the $n^{\mathrm{th}}$ power with $T$, we multiply with the truncated $\overrightarrow{\AS}^*$ row vector: 
		\begin{align}
			S:=\begin{bmatrix}
				s^6 & s^4 & 3s^6
			\end{bmatrix}\,.
		\end{align}
		In the case of $\AS(1)$, we simply compute $ST=4s^{12}q^2+4s^{12}q^2+24s^{12}\sim 8s^{12}q^2$ and find our result is consistent with Remark \ref{rem.AS1}. 
		
		An inductive argument then shows that the $(s,q)$-leading terms of each component of $T\boxtimes T$ are given by 
		\begin{align}
			\renewcommand{\arraystretch}{1.25}
			\begin{bmatrix}
				T^{\boxtimes n}_{\llsub} \\T^{\boxtimes n}_{\ccsub}\\T^{\boxtimes n}_{\xxsub}
			\end{bmatrix}
			\boxtimes
			T
			\sim
			\begin{bmatrix}
				4s^8q^2\cdot T^{\boxtimes n}_{\llsub} + 4s^6q^2\cdot T^{\boxtimes n}_{\ccsub}
				\\
				4s^8q^2\cdot T^{\boxtimes n}_{\ccsub}
				\\
				8s^6\cdot T^{\boxtimes n}_{\ccsub}+4s^8q^2\cdot T^{\boxtimes n}_{\xxsub}
			\end{bmatrix}\,.
		\end{align}
		The $s$- and $q$-degrees of each term in $T^{\boxtimes n}$ are $s^{8n-2}$, $s^{8n}$, $s^{8n-2}$ and $q^{2+2n}$, $q^{2+2n}$, $q^{2n}$. This shows that 
		the contribution of the $T^{\boxtimes n}_{\xxsub}$ component to the $(s,q)$-leading term in 
		\begin{align}
		    S\cdot \left(\begin{bmatrix}
			T^{\boxtimes n}_{\llsub} \\T^{\boxtimes n}_{\ccsub}\\T^{\boxtimes n}_{\xxsub}
		\end{bmatrix}
		\boxtimes
		T\right)
		\end{align}
		is lesser by a factor of $q^2$. Hence, it suffices to consider the coupled sequence:
		\begin{align}
			\begin{cases}
				T^{\boxtimes (n+1)}_{\llsub} = 4s^8q^2\cdot T^{\boxtimes n}_{\llsub} + 4s^6q^2\cdot T^{\boxtimes n}_{\ccsub},
				&
				T^{\boxtimes 1}_{\llsub}=4s^6q^2
				\\
				T^{\boxtimes (n+1)}_{\ccsub} = 4s^8q^2\cdot T^{\boxtimes n}_{\ccsub},
				&
				T^{\boxtimes 1}_{\ccsub}=4s^8q^2
			\end{cases}\,.
		\end{align}
		The latter sequence is given explicitly by $T^{\boxtimes n}_{\ccsub}=(4s^8q^2)^n$ and then the former by $T^{\boxtimes n}_{\llsub}=ns^{-2}(4s^8q^2)^n$. Now substituting into $\LG_{\AS(n)}(s,q)\sim s^6\cdot T^{\boxtimes n}_{\llsub}+s^4\cdot T^{\boxtimes n}_{\ccsub}$ produces the desired formula.
	\end{proof}
	
	The $s$-span of the Links-Gould polynomial of a link $L$ with $\mu$ components provides a lower bound for the genus of any compact oriented surface embedded in $S^3$ with no closed components and whose boundary is $L$, see \cite{Kohli-Tahar, LNVdV25}: 
	\begin{align}
	    \mathrm{span}_s(\LG_L(s,q)) \leq 4 \cdot (2~\mathrm{genus}(L) + \mu -1).
	\end{align}
	
	Along with the Seifert surface obtained from the Seifert algorithm applied to the diagram for $\AS(n)$ shown in Figure~\ref{fig:FIG1}, we determine the genus of the AS links.

    \begin{figure}[h!]
		\begin{center}
			\adjustbox{trim=0cm 0cm 0cm 0cm}{
				\includegraphics[height=.4\textheight]{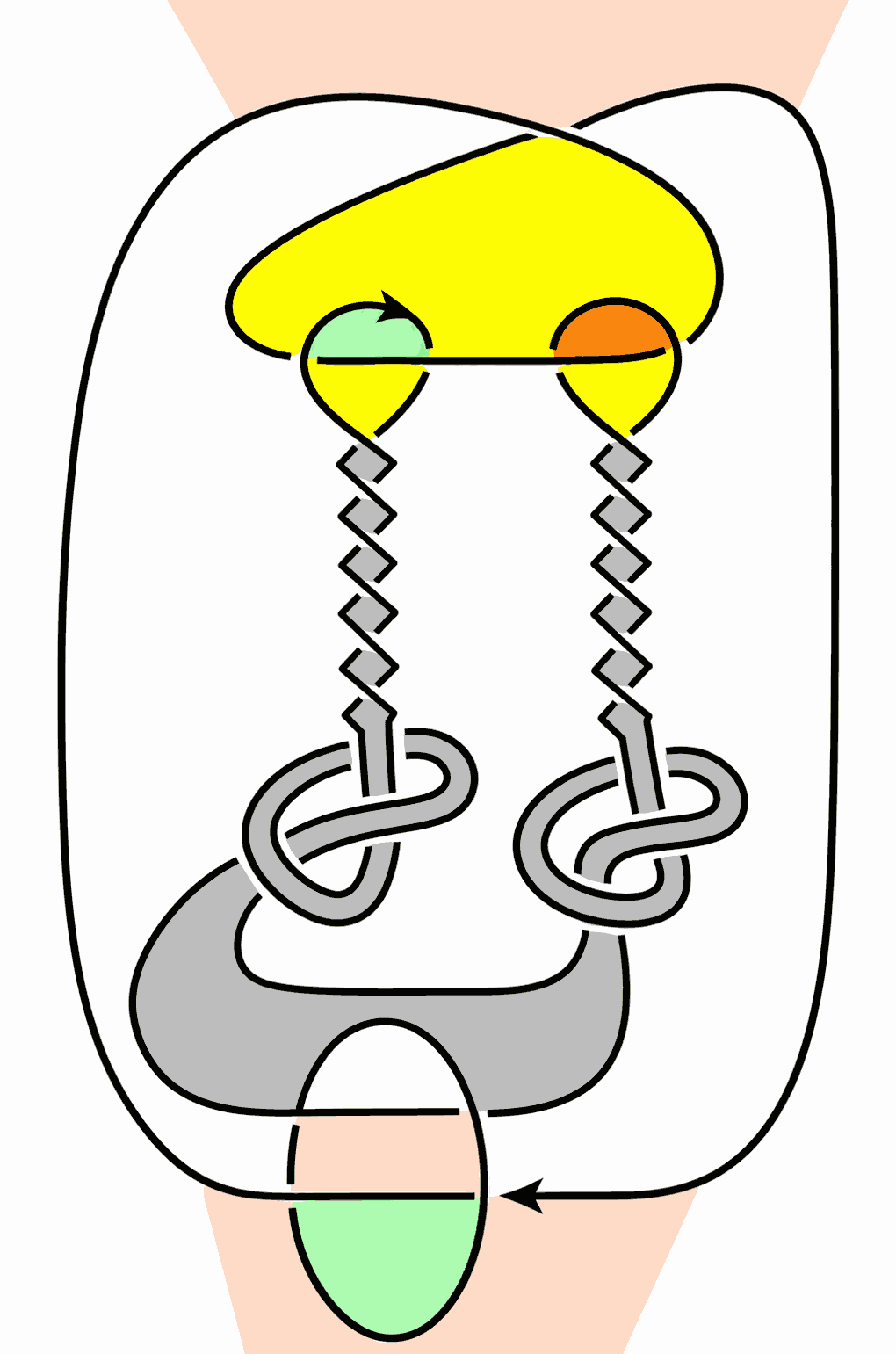}
			}
			\caption{A disk-band decomposition of a Seifert surface for $\AS(1)$ with $6$ disks, $11$ handles and $3$ boundary components.} \label{fig:FIG1}
		\end{center}
	\end{figure} 
    
	\begin{corollary}
		For all $n \in \mathbb{Z}_{\geq 1}$, $\mathrm{genus}(\AS(n)) = 2 n.$
	\end{corollary}
	
	\begin{proof}
		If $\Sigma$ is a Seifert surface for $L$ a link with $\mu$ components, set $k$ the number of disks in a disk-handle decomposition of $\Sigma$, and $l$ the number of $1$-handles. An Euler characteristic computation proves that for connected $\Sigma$
		\begin{align}
		    2~\mathrm{genus}(\Sigma) + \mu -1 = 1 - \chi(\Sigma) = 1 - k + l.
		\end{align}
		
		The Seifert surface $\Sigma_1$ for $\AS(1)$ represented in Figure~\ref{fig:FIG1} has $k = 6$ disks and $l = 11$ $1$-handles. So for this surface $1 - \chi(\Sigma_1) = 1 - 6 + 11 = 6$. Moreover, when passing from $\AS(n)$ to $\AS(n+1)$, $6$ handles and $2$ disks are added: an extra green disk and an extra orange disk (see Figure~\ref{fig:FIG1}) link the yellow and grey disks, adding two handles between the new green disk and the yellow one, two handles between the new orange disk and that same yellow disk, and finally two handles between the preexisting yellow and grey disks. Therefore $1 - \chi(\Sigma_n) = 4n + 2$, where $\Sigma_n$ is the natural Seifert surface from the diagram of $\AS(n)$. Along with the lower bound given by the $s$-span of $\LG$, this implies that the minimal value for $1 - \chi(\Sigma_n')$, where $\Sigma_n'$ is any Seifert surface for $\AS(n)$, is $4 n + 2$. So the genus of $\AS(n)$ is
		\begin{align}
		    \mathrm{genus}(\AS(n)) = \frac{(1-\chi)_{min}+1-\mu}{2} = \frac{4n + 2 + 1 - 3}{2} = 2 n.
		\end{align}
	\end{proof}
    The Alexander polynomial, which fails to detect any Allen-Swenberg link from $H\# H$, gives no genus information here, so the gap between the two bounds grows linearly in $n$.
    
    \appendix
    \section{Matrix coefficients of the double-braiding}\label{app.doublebraid}
	   For the interested reader, we give a description of the double twist map on $V(0,\a)\otimes V(0,\a)^*$. Following in the notation of Section \ref{sec.LG}, we specify the image of the vector $x_i\otimes y_j$ in matrix coordinate $(i,j)$. We suppress the tensor product symbol from the output notation.
	\begin{tiny}
		\begin{align*}
			\xx_{11} &= 3 x_{0} y_{0} + 2 x_{1} y_{1} + 2 x_{2} y_{2} + x_{3} y_{3}
			- \frac{2 x_{0} y_{0} + 2 x_{1} y_{1} + 2 x_{2} y_{2} + x_{3} y_{3}
				+ \frac{2 x_{0} y_{0}}{q^{2}} + \frac{x_{1} y_{1}}{q^{2}} + \frac{x_{2} y_{2}}{q^{2}} + \frac{x_{3} y_{3}}{q^{2}}}{s^{2}} \\
			&\quad + \frac{\frac{x_{0} y_{0}}{q^{2}} + \frac{x_{1} y_{1}}{q^{2}} + \frac{x_{2} y_{2}}{q^{2}} + \frac{x_{3} y_{3}}{q^{2}}}{s^{4}}
			+ \frac{x_{0} y_{0}}{q^{2}} \, ,
			\\
			\xx_{12} &= \frac{q^{2} s^{2} x_{0} y_{1} + q^{2} s^{2} x_{2} y_{3} + s^{2} x_{0} y_{1} - x_{0} y_{1} - x_{2} y_{3}}{q^{2} s^{2}} \, ,\\
			\xx_{13} &= \frac{q^{3} s^{2} x_{0} y_{2} - q^{2} s^{2} x_{1} y_{3} + q s^{2} x_{0} y_{2} - q x_{0} y_{2} + x_{1} y_{3}}{q^{3} s^{2}} \, ,\\
			\xx_{14} &= x_{0} y_{3} \, ,\\
			\xx_{21} &= \frac{q^{2} s^{2} x_{1} y_{0} + q^{2} s^{2} x_{3} y_{2} - q^{2} x_{3} y_{2} + s^{2} x_{1} y_{0} - x_{1} y_{0}}{q^{2} s^{2}} \, ,\\
			\xx_{22} &= - \frac{2 q^{2} s^{4} x_{0} y_{0} + q^{2} s^{4} x_{1} y_{1} + q^{2} s^{4} x_{2} y_{2}
				- 2 q^{2} s^{2} x_{0} y_{0} - 2 q^{2} s^{2} x_{1} y_{1} - 2 q^{2} s^{2} x_{2} y_{2}
				- q^{2} s^{2} x_{3} y_{3}}{q^{2} s^{4}}  \\
			&\quad {} - \frac{- s^{4} x_{1} y_{1} - s^{2} x_{0} y_{0}
				+ x_{0} y_{0} + x_{1} y_{1} + x_{2} y_{2} + x_{3} y_{3}}{q^{2} s^{4}} \, ,
			\\
			\xx_{23} &= \frac{x_{1} y_{2}}{q^{2}} \, ,\\
			\xx_{24} &= \frac{q s^{2} x_{0} y_{2} - q x_{0} y_{2} + x_{1} y_{3}}{q^{2} s^{2}} \, ,\\
			\xx_{31} &= \frac{q^{2} s^{2} x_{2} y_{0} - q s^{2} x_{3} y_{1} + q x_{3} y_{1} + s^{2} x_{2} y_{0} - x_{2} y_{0}}{q^{2} s^{2}} \, ,\\
			\xx_{32} &= \frac{x_{2} y_{1}}{q^{2}} \, ,\\
			\xx_{33}&=- \frac{2 q^{2} s^{4} x_{0} y_{0} + q^{2} s^{4} x_{1} y_{1} - 2 q^{2} s^{2} x_{0} y_{0} - 2 q^{2} s^{2} x_{1} y_{1} - 2 q^{2} s^{2} x_{2} y_{2} - q^{2} s^{2} x_{3} y_{3} - s^{2} x_{0} y_{0} + x_{0} y_{0} + x_{1} y_{1} + x_{2} y_{2} + x_{3} y_{3}}{q^{4} s^{4}} \, ,
			\\
			\xx_{34} &= - \frac{s^{2} x_{0} y_{1} - x_{0} y_{1} - x_{2} y_{3}}{q^{2} s^{2}} \, ,\\
			\xx_{41} &= x_{3} y_{0} \, ,\\
			\xx_{42} &= \frac{q^{2} s^{2} x_{2} y_{0} + q x_{3} y_{1} - x_{2} y_{0}}{q^{3} s^{2}} \, ,\\
			\xx_{43} &= - \frac{q^{2} s^{2} x_{1} y_{0} - q^{2} x_{3} y_{2} - x_{1} y_{0}}{q^{4} s^{2}} \, ,\\
			\xx_{44}&=\frac{q^{2} s^{4} x_{0} y_{0} - q^{2} s^{2} x_{0} y_{0} - q^{2} s^{2} x_{1} y_{1} - q^{2} s^{2} x_{2} y_{2} - s^{2} x_{0} y_{0} + x_{0} y_{0} + x_{1} y_{1} + x_{2} y_{2} + x_{3} y_{3}}{q^{4} s^{4}} \, .
		\end{align*}
	\end{tiny}
	
	\bibliographystyle{alpha}
	\bibliography{biblio}

@preamble{
   "\def\cprime{$'$} "
}

@article {AS,
    AUTHOR = {Allen, Samantha and Swenberg, Jacob H.},
     TITLE = {Do link polynomials detect causality in globally hyperbolic
              spacetimes?},
   JOURNAL = {J. Math. Phys.},
  FJOURNAL = {Journal of Mathematical Physics},
    VOLUME = {62},
      YEAR = {2021},
    NUMBER = {3},
     PAGES = {Paper No. 032503, 16},
      ISSN = {0022-2488,1089-7658},
   MRCLASS = {57Z05 (57K14)},
  MRNUMBER = {4227566},
MRREVIEWER = {Xin\ Liu},
       DOI = {10.1063/5.0040956},
       URL = {https://doi.org/10.1063/5.0040956},
}

@misc{ArnoldProblem,
  author = {Arnold, Vladimir I.},
  title = {Problems},
  note = {Written down by S. Duzhin, September 1998},
  url = {http://www.pdmi.ras.ru/~arnsem/Arnold/prob9809.ps.gz}
}

@book {ArnoldProblemBook,
    AUTHOR = {Arnold, Vladimir I.},
     TITLE = {Arnold's problems},
   EDITION = {revised},
      NOTE = {With a preface by V. Philippov, A. Yakivchik and M. Peters},
 PUBLISHER = {Springer-Verlag, Berlin; PHASIS, Moscow},
      YEAR = {2004},
     PAGES = {xvi+639},
      ISBN = {3-540-20614-0},
   MRCLASS = {58-02 (00A07 01A72 37-02 53-02 57-02)},
  MRNUMBER = {2078115},
}

@misc{Bax,
  author = {Baxshilloyev, Amirbek},
  title = {Causality {D}etection via {S}ymplectic {Q}uandles},
  year = {2025},
  note = {Preprint 2025, \href{https://arxiv.org/abs/2508.18323}{arXiv:2508.18323}}
}

@article {BeSaCausal,
    AUTHOR = {Bernal, Antonio N. and S\'anchez, Miguel},
     TITLE = {Globally hyperbolic spacetimes can be defined as `causal'
              instead of `strongly causal'},
   JOURNAL = {Classical Quantum Gravity},
  FJOURNAL = {Classical and Quantum Gravity},
    VOLUME = {24},
      YEAR = {2007},
    NUMBER = {3},
     PAGES = {745--749},
      ISSN = {0264-9381,1361-6382},
   MRCLASS = {53C50 (53C80)},
  MRNUMBER = {2294243},
MRREVIEWER = {Paul\ E.\ Ehrlich},
       DOI = {10.1088/0264-9381/24/3/N01},
       URL = {https://doi.org/10.1088/0264-9381/24/3/N01},
}

@article {BeSa1,
    AUTHOR = {Bernal, Antonio N. and S\'anchez, Miguel},
     TITLE = {On smooth {C}auchy hypersurfaces and {G}eroch's splitting
              theorem},
   JOURNAL = {Comm. Math. Phys.},
  FJOURNAL = {Communications in Mathematical Physics},
    VOLUME = {243},
      YEAR = {2003},
    NUMBER = {3},
     PAGES = {461--470},
      ISSN = {0010-3616,1432-0916},
   MRCLASS = {53C50 (53C80 83C05)},
  MRNUMBER = {2029362},
MRREVIEWER = {Paul\ E.\ Ehrlich},
       DOI = {10.1007/s00220-003-0982-6},
       URL = {https://doi.org/10.1007/s00220-003-0982-6},
}

@article {BeSa2,
    AUTHOR = {Bernal, Antonio N. and S\'anchez, Miguel},
     TITLE = {Smoothness of time functions and the metric splitting of
              globally hyperbolic spacetimes},
   JOURNAL = {Comm. Math. Phys.},
  FJOURNAL = {Communications in Mathematical Physics},
    VOLUME = {257},
      YEAR = {2005},
    NUMBER = {1},
     PAGES = {43--50},
      ISSN = {0010-3616,1432-0916},
   MRCLASS = {53C50 (83C20)},
  MRNUMBER = {2163568},
MRREVIEWER = {Paul\ E.\ Ehrlich},
       DOI = {10.1007/s00220-005-1346-1},
       URL = {https://doi.org/10.1007/s00220-005-1346-1},
}

@article {BeSa3,
    AUTHOR = {Bernal, Antonio N. and S\'anchez, Miguel},
     TITLE = {Further results on the smoothability of {C}auchy hypersurfaces
              and {C}auchy time functions},
   JOURNAL = {Lett. Math. Phys.},
  FJOURNAL = {Letters in Mathematical Physics},
    VOLUME = {77},
      YEAR = {2006},
    NUMBER = {2},
     PAGES = {183--197},
      ISSN = {0377-9017,1573-0530},
   MRCLASS = {53C50 (53C80 81T20)},
  MRNUMBER = {2254187},
MRREVIEWER = {Paul\ E.\ Ehrlich},
       DOI = {10.1007/s11005-006-0091-5},
       URL = {https://doi.org/10.1007/s11005-006-0091-5},
}

@book {Besse,
    AUTHOR = {Besse, Arthur L.},
     TITLE = {Manifolds all of whose geodesics are closed},
    SERIES = {Ergebnisse der Mathematik und ihrer Grenzgebiete [Results in
              Mathematics and Related Areas]},
    VOLUME = {93},
      NOTE = {With appendices by D. B. A. Epstein, J.-P. Bourguignon, L.
              B\'erard-Bergery, M. Berger and J. L. Kazdan},
 PUBLISHER = {Springer-Verlag, Berlin-New York},
      YEAR = {1978},
     PAGES = {ix+262},
      ISBN = {3-540-08158-5},
   MRCLASS = {53C20 (53C22 58G99)},
  MRNUMBER = {496885},
MRREVIEWER = {R.\ L.\ Bishop},
}

@misc{Chen,
  author = {Chen, Hongxu},
  title = {Medial quandles's capability of detecting causality and properties of their coloring on certain links and knot},
  year = {2024},
  note = {Preprint 2024, \href{https://arxiv.org/abs/2411.04477}{arXiv:2411.04477}}
}

@article {CHigherDimensions,
    AUTHOR = {Chernov, Vladimir},
     TITLE = {Causality and {L}egendrian linking for higher dimensional
              spacetimes},
   JOURNAL = {J. Geom. Phys.},
  FJOURNAL = {Journal of Geometry and Physics},
    VOLUME = {133},
      YEAR = {2018},
     PAGES = {26--29},
      ISSN = {0393-0440,1879-1662},
   MRCLASS = {53C50 (57M25 57R17)},
  MRNUMBER = {3850254},
MRREVIEWER = {Marco\ Golla},
       DOI = {10.1016/j.geomphys.2018.06.018},
       URL = {https://doi.org/10.1016/j.geomphys.2018.06.018},
}

@article {CMP,
    AUTHOR = {Chernov, Vladimir and Martin, Gage and Petkova, Ina},
     TITLE = {Khovanov homology and causality in spacetimes},
   JOURNAL = {J. Math. Phys.},
  FJOURNAL = {Journal of Mathematical Physics},
    VOLUME = {61},
      YEAR = {2020},
    NUMBER = {2},
     PAGES = {022503, 3},
      ISSN = {0022-2488,1089-7658},
   MRCLASS = {57K18 (83C05)},
  MRNUMBER = {4067131},
MRREVIEWER = {Giuseppe\ Nardelli},
       DOI = {10.1063/5.0002297},
       URL = {https://doi.org/10.1063/5.0002297},
}

@article {CNGAFA,
    AUTHOR = {Chernov, Vladimir and Nemirovski, Stefan},
     TITLE = {Legendrian links, causality, and the {L}ow conjecture},
   JOURNAL = {Geom. Funct. Anal.},
  FJOURNAL = {Geometric and Functional Analysis},
    VOLUME = {19},
      YEAR = {2010},
    NUMBER = {5},
     PAGES = {1320--1333},
      ISSN = {1016-443X,1420-8970},
   MRCLASS = {53C50 (53D10 57M25 57R17)},
  MRNUMBER = {2585576},
MRREVIEWER = {Paul\ E.\ Ehrlich},
       DOI = {10.1007/s00039-009-0039-x},
       URL = {https://doi.org/10.1007/s00039-009-0039-x},
}

@article {CNGT,
    AUTHOR = {Chernov, Vladimir and Nemirovski, Stefan},
     TITLE = {Non-negative {L}egendrian isotopy in {$ST^*M$}},
   JOURNAL = {Geom. Topol.},
  FJOURNAL = {Geometry \& Topology},
    VOLUME = {14},
      YEAR = {2010},
    NUMBER = {1},
     PAGES = {611--626},
      ISSN = {1465-3060,1364-0380},
   MRCLASS = {53D35},
  MRNUMBER = {2602847},
MRREVIEWER = {Hansj\"org\ Geiges},
       DOI = {10.2140/gt.2010.14.611},
       URL = {https://doi.org/10.2140/gt.2010.14.611},
}

@article {CRCMP,
    AUTHOR = {Chernov, Vladimir and Rudyak, Yuli},
     TITLE = {Linking and causality in globally hyperbolic space-times},
   JOURNAL = {Comm. Math. Phys.},
  FJOURNAL = {Communications in Mathematical Physics},
    VOLUME = {279},
      YEAR = {2008},
    NUMBER = {2},
     PAGES = {309--354},
      ISSN = {0010-3616,1432-0916},
   MRCLASS = {53C50 (53C80 57Q45 57R17 83C75)},
  MRNUMBER = {2383590},
MRREVIEWER = {Robert\ J.\ Low},
       DOI = {10.1007/s00220-008-0414-8},
       URL = {https://doi.org/10.1007/s00220-008-0414-8},
}

@misc{Fan,
  author = {Fan, Zining},
  title = {Detecting {C}ausality with {C}onjugation {Q}uandles over {D}ihedral {G}roups},
  year = {2025},
  note = {Preprint 2025, \href{https://arxiv.org/abs/2509.03544}{arXiv:2509.03544}}
}

@article {Geroch,
    AUTHOR = {Geroch, Robert},
     TITLE = {Domain of dependence},
   JOURNAL = {J. Mathematical Phys.},
  FJOURNAL = {Journal of Mathematical Physics},
    VOLUME = {11},
      YEAR = {1970},
     PAGES = {437--449},
      ISSN = {0022-2488,1089-7658},
   MRCLASS = {83.35},
  MRNUMBER = {270697},
MRREVIEWER = {W.\ Israel},
       DOI = {10.1063/1.1665157},
       URL = {https://doi.org/10.1063/1.1665157},
}

@Misc{Github,
    title = {{LG}-{A}llen-{S}wenberg},
    author = {Chernov, Vladimir and Harper, Matthew and Kohli, Ben-Michael},
    year = {2026},
    note = {\url{https://github.com/mrhmath/LG-Allen-Swenberg/}},
}

@misc{LG-tangle-computer,
  title = {{LG}-tangle-computer},
  author = {Harper, Matthew},
  year = {2026},
  note = {\url{https://github.com/mrhmath/LG-tangle-computer}},
}

@book {HE,
    AUTHOR = {Hawking, Stephen and Ellis, George},
     TITLE = {The large scale structure of space-time},
    SERIES = {Cambridge Monographs on Mathematical Physics},
    VOLUME = {No. 1},
 PUBLISHER = {Cambridge University Press, London-New York},
      YEAR = {1973},
     PAGES = {xi+391},
   MRCLASS = {83.58},
  MRNUMBER = {424186},
MRREVIEWER = {Michael\ P.\ Ryan, Jr.},
}

@article {Jain,
    AUTHOR = {Jain, Ayush},
     TITLE = {Detecting causality with symplectic quandles},
   JOURNAL = {Lett. Math. Phys.},
  FJOURNAL = {Letters in Mathematical Physics},
    VOLUME = {114},
      YEAR = {2024},
    NUMBER = {3},
     PAGES = {Paper No. 63, 22},
      ISSN = {0377-9017,1573-0530},
   MRCLASS = {57K12 (57K10 57K14)},
  MRNUMBER = {4742225},
MRREVIEWER = {Mohamed\ Elhamdadi},
       DOI = {10.1007/s11005-024-01808-w},
       URL = {https://doi.org/10.1007/s11005-024-01808-w},
}

@article {Khovanov,
    AUTHOR = {Khovanov, Mikhail},
     TITLE = {A categorification of the {J}ones polynomial},
   JOURNAL = {Duke Math. J.},
  FJOURNAL = {Duke Mathematical Journal},
    VOLUME = {101},
      YEAR = {2000},
    NUMBER = {3},
     PAGES = {359--426},
      ISSN = {0012-7094,1547-7398},
   MRCLASS = {57M27 (57R56)},
  MRNUMBER = {1740682},
       DOI = {10.1215/S0012-7094-00-10131-7},
       URL = {https://doi.org/10.1215/S0012-7094-00-10131-7},
       note = {Preprint 2000, \href{https://arxiv.org/abs/math/9908171}{arXiv:math/9908171}},
}

@misc{Leventhal,
  author = {Leventhal, Jack},
  title = {Alexander {Q}uandles and {D}etecting {C}ausality},
  year = {2022},
  note = {Preprint 2022, \href{https://arxiv.org/abs/2209.05670}{arXiv:2209.05670}}
}

@phdthesis{Low0,
  author = {Low, Robert},
  title = {Causal relations and spaces of null geodesics},
  school = {University of Oxford},
  year = {1988}
}

@article {Low1,
    AUTHOR = {Low, Robert},
     TITLE = {Twistor linking and causal relations},
   JOURNAL = {Classical Quantum Gravity},
  FJOURNAL = {Classical and Quantum Gravity},
    VOLUME = {7},
      YEAR = {1990},
    NUMBER = {2},
     PAGES = {177--187},
      ISSN = {0264-9381,1361-6382},
   MRCLASS = {83C60 (53C50 53C80 57Q45)},
  MRNUMBER = {1042879},
       DOI = {10.1088/0264-9381/7/2/011},
       URL = {https://doi.org/10.1088/0264-9381/7/2/011},
}

@article {Low3,
    AUTHOR = {Low, Robert},
     TITLE = {Twistor linking and causal relations in exterior
              {S}chwarzschild space},
   JOURNAL = {Classical Quantum Gravity},
  FJOURNAL = {Classical and Quantum Gravity},
    VOLUME = {11},
      YEAR = {1994},
    NUMBER = {2},
     PAGES = {453--456},
      ISSN = {0264-9381,1361-6382},
   MRCLASS = {83C60 (32L25 57M25)},
  MRNUMBER = {1259156},
MRREVIEWER = {Jorge\ Goncalves\ Cardoso},
       DOI = {10.1088/0264-9381/11/2/016},
       URL = {https://doi.org/10.1088/0264-9381/11/2/016},
}

@article {LowLegendrian,
    AUTHOR = {Low, Robert},
     TITLE = {Stable singularities of wave-fronts in general relativity},
   JOURNAL = {J. Math. Phys.},
  FJOURNAL = {Journal of Mathematical Physics},
    VOLUME = {39},
      YEAR = {1998},
    NUMBER = {6},
     PAGES = {3332--3335},
      ISSN = {0022-2488,1089-7658},
   MRCLASS = {83C10 (58C27 58F05 78A05)},
  MRNUMBER = {1623606},
MRREVIEWER = {Volker\ Perlick},
       DOI = {10.1063/1.532257},
       URL = {https://doi.org/10.1063/1.532257},
}

@article {NatarioTod,
    AUTHOR = {Nat\'ario, Jos\'e{} and Tod, Paul},
     TITLE = {Linking, {L}egendrian linking and causality},
   JOURNAL = {Proc. London Math. Soc. (3)},
  FJOURNAL = {Proceedings of the London Mathematical Society. Third Series},
    VOLUME = {88},
      YEAR = {2004},
    NUMBER = {1},
     PAGES = {251--272},
      ISSN = {0024-6115,1460-244X},
   MRCLASS = {53D10 (57M25 58K20)},
  MRNUMBER = {2018966},
MRREVIEWER = {Robert\ J.\ Low},
       DOI = {10.1112/S0024611503014424},
       URL = {https://doi.org/10.1112/S0024611503014424},
}

@article{HF1,
  author = {Ozsvath, Peter and Szabo, Zoltan},
  title = {Holomorphic disks and three-manifold invariants: properties and applications},
  journal = {Annals of Mathematics},
  volume = {159},
  number = {3},
  pages = {1159--1245},
  year = {2004}
}

@article{HF2,
  author = {Ozsvath, Peter and Szabo, Zoltan},
  title = {Holomorphic disks and topological invariants for closed three-manifolds},
  journal = {Annals of Mathematics},
  volume = {159},
  number = {3},
  pages = {1027--1158},
  year = {2004}
}

@incollection{Penrose,
  author = {Penrose, Roger},
  title = {The question of cosmic censorship},
  booktitle = {Black Holes and Relativistic Stars},
  publisher = {University of Chicago Press},
  address = {Chicago},
  pages = {103--122},
  year = {1998}
}

@Misc {GHKW,
    AUTHOR = {Garoufalidis, Stavros and Harper, Matthew and Kashaev, Rinat and Kohli,
             Ben-Michael and Wagner, Emmanuel},
     TITLE = {On the colored {L}inks--{G}ould polynomial},
      NOTE = {Preprint 2025,
             \href{https://arxiv.org/abs/2509.10911}{arXiv:2509.10911}},
}

@article {GeerKujawaPatureau,
    AUTHOR = {Geer, Nathan and Kujawa, Jonathan and Patureau-Mirand,
              Bertrand},
     TITLE = {Generalized trace and modified dimension functions on ribbon
              categories},
   JOURNAL = {Selecta Math. (N.S.)},
  FJOURNAL = {Selecta Mathematica. New Series},
    VOLUME = {17},
      YEAR = {2011},
    NUMBER = {2},
     PAGES = {453--504},
      ISSN = {1022-1824,1420-9020},
   MRCLASS = {18D10 (17B10 20C20 57R56)},
  MRNUMBER = {2803849},
MRREVIEWER = {Daniel\ David\ Moskovich},
       DOI = {10.1007/s00029-010-0046-7},
       URL = {https://doi.org/10.1007/s00029-010-0046-7},
}

@article {GPT,
    AUTHOR = {Geer, Nathan and Patureau-Mirand, Bertrand and Turaev,
              Vladimir},
     TITLE = {Modified quantum dimensions and re-normalized link invariants},
   JOURNAL = {Compos. Math.},
  FJOURNAL = {Compositio Mathematica},
    VOLUME = {145},
      YEAR = {2009},
    NUMBER = {1},
     PAGES = {196--212},
      ISSN = {0010-437X,1570-5846},
   MRCLASS = {57M27 (17B37 18D10)},
  MRNUMBER = {2480500},
MRREVIEWER = {Daniel\ David\ Moskovich},
       DOI = {10.1112/S0010437X08003795},
       URL = {https://doi.org/10.1112/S0010437X08003795},
}

@Misc{HKST,
  AUTHOR = {Harper, Matthew and Kohli, Ben-Michael and Song, Jiebo and Tahar, Guillaume},
  TITLE = {On some log-concavity properties of the {Alexander-Conway} and {Links-Gould} invariants},
  NOTE = {Preprint 2025,
          \href{https://arxiv.org/abs/2509.16868}{arXiv:2509.16868}},
}

@article {Ishii,
    AUTHOR = {Ishii, Atsushi},
     TITLE = {The {L}inks-{G}ould polynomial as a generalization of the
              {A}lexander-{C}onway polynomial},
   JOURNAL = {Pacific J. Math.},
  FJOURNAL = {Pacific Journal of Mathematics},
    VOLUME = {225},
      YEAR = {2006},
    NUMBER = {2},
     PAGES = {273--285},
      ISSN = {0030-8730},
   MRCLASS = {57M27 (57M25)},
  MRNUMBER = {2233736},
MRREVIEWER = {Nafaa Chbili},
       DOI = {10.2140/pjm.2006.225.273},
       URL = {https://doi.org/10.2140/pjm.2006.225.273},
}

@article{Kohli,
  title={On the {L}inks--{G}ould invariant and the square of the {A}lexander polynomial},
  author={Kohli, Ben-Michael},
  journal={Journal of Knot Theory and Its Ramifications},
  volume={25},
  number={02},
  pages={1650006},
  year={2016},
  publisher={World Scientific}
}

@article{KPM,
  title={Other quantum relatives of the {A}lexander polynomial through the {L}inks-{G}ould invariants},
  author={Kohli, Ben-Michael and Patureau-Mirand, Bertrand},
  journal={Proceedings of the American Mathematical Society},
  volume={145},
  number={12},
  pages={5419--5433},
  year={2017}
}

@article {Kohli-Tahar,
    AUTHOR = {Kohli, Ben-Michael and Tahar, Guillaume},
     TITLE = {A lower bound for the genus of a knot using the {L}inks--{G}ould invariant},
   JOURNAL = {Quantum Topol.},
  FJOURNAL = {Quantum Topology},
      YEAR = {2026},
      ISSN = {1663-487X},
   MRCLASS = {57K10 (17B37 57K16)},
       DOI = {10.4171/qt/263},
       URL = {https://doi.org},
      NOTE = {Published online first},
}

@misc{LNVdV25,
  title={A plumbing-multiplicative function from the Links-Gould invariant},
  author={Lopez-Neumann, Daniel and van der Veen, Roland},
NOTE = {Preprint 2025,
             \href{https://arxiv.org/abs/2502.12899}{arXiv:2502.12899}},
}

@misc{WilsonQuandle,
  author = {Wilson, Nathaniel},
  title = {CAPTURING CAUSALITY IN GLOBALLY HYPERBOLIC SPACETIMES WITH QUANDLES},
  note = {submitted to arXiv and currently on hold, 20 pages},
  year = {2026}
}
	
\end{document}